\documentclass[12pt]{amsart}

\usepackage{amsmath,amscd}
\usepackage{amssymb}
\usepackage{amsthm}

\usepackage{mathrsfs}
\usepackage{hyperref}

\usepackage{calc}
             {\begin{list}{\arabic{enumi}.}{\usecounter{enumi}%
              \setlength{\labelsep}{0.5em}%
              \settowidth{\labelwidth}{\arabic{enumi}.}%
              \setlength{\leftmargin}{\labelwidth+\labelsep}}}%
             {\end{list}}

\newtheorem{Theorem}{Theorem}[section]

\newtheorem{Lemma}[Theorem]{Lemma}

\theoremstyle{definition}
\newtheorem{Definition}{Definition}[section]

\newtheorem*{ExampleNoNumber}{Example}

\newtheorem*{RemarkNoNumber}{Remark}

\newtheorem*{QuestionNoNumber}{Question}

\numberwithin{equation}{section}

\newcommand{\mR}{\mathbb{R}}                    
\newcommand{\abs}[1]{\lvert #1 \rvert}          
\newcommand{\norm}[1]{\lVert #1 \rVert}         

\newcommand{\eps}{\varepsilon}

\newcommand{\mO}{\mathcal{O}}

\newcounter{sidenote}
\begin{document}

\title{The Calder{\'o}n problem and normal forms}

\author[M. Salo]{Mikko Salo}
\address{Department of Mathematics and Statistics, University of Jyv\"askyl\"a}
\email{mikko.j.salo@jyu.fi}


\date{\today}


\begin{abstract}
We outline an approach to the inverse problem of Calder\'on that highlights the role of microlocal normal forms and propagation of singularities and extends a number of earlier results also in the anisotropic case. The main result states that from the boundary measurements it is possible to recover  integrals of the unknown coefficient over certain two-dimensional manifolds called good bicharacteristic leaves. This reduces the Calder\'on problem into solving a linear integral geometry problem (inversion of a bicharacteristic leaf transform).
\end{abstract}

\maketitle

\section{Introduction} \label{sec_intro}

The inverse conductivity problem posed by Calder\'on \cite{C} asks to determine the electrical conductivity of a medium from measurements of electrical voltage and current on its boundary. This question is the mathematical model for Electrical Resistivity Tomography and Electrical Impedance Tomography, which are imaging methods that have applications in geophysical, industrial and medical imaging. The problem has a central role in the mathematical theory of inverse problems: it provides a model case for various inverse problems and imaging methods modelled by elliptic equations, including optical or acoustic tomography, and has interesting connections to other topics such as inverse scattering theory, geometric rigidity problems, and invisibility studies. We refer to \cite{U14} for further information and references to the substantial literature on this problem.

The Calder\'on problem in two dimensions is relatively well understood, but there are a number of open questions in dimensions $\geq 3$ including the case of matrix-valued coefficients (anisotropic Calder\'on problem) and partial data. In this work we will consider a variant of the anisotropic Calder\'on problem as in \cite{DKSaU, DKLS} where the unknown coefficient is a lower order term and difficulties related to diffeomorphism invariance go away. This article is an announcement of the results with sketches of proofs. Detailed proofs will appear in a later version, where also the final form of the results may slightly differ.

It is well known that in dimensions $\geq 3$, it is convenient to state the question using the language of Riemannian geometry. Let $(M,g)$ be a compact oriented Riemannian $n$-manifold with smooth boundary, and let $q \in C(M)$ (we assume $q$ continuous for simplicity so that the integral transforms below are well defined). We denote by $\Delta_g$ the Laplace-Beltrami operator on $(M,g)$, and consider boundary measurements for the Schr\"odinger equation 
\[
(-\Delta_g + q)u = 0 \text{ in $M$}
\]
given by the Cauchy data set (where $\partial_{\nu}$ is the normal derivative) 
\[
C_{g,q} = \{ (u|_{\partial M}, \partial_{\nu} u|_{\partial M}) \,;\, u \in H^1(M), \ (-\Delta_g+q)u = 0 \text{ in $M$} \}.
\]
If $0$ is not a Dirichlet eigenvalue for $-\Delta_g+q$ in $M$, then knowing $C_{g,q}$ is equivalent to knowing the more traditional boundary measurements given by the Dirichlet-to-Neumann map $\Lambda_q$ mapping a Dirichlet data on $\partial M$ to the corresponding Neumann data of the solution on $\partial M$.

We will consider uniqueness in the Calder\'on problem when $(M,g)$ is known and the potential $q$ is unknown:

\begin{QuestionNoNumber}
Let $(M,g)$ be a compact oriented Riemannian manifold with smooth boundary, and let $q_1, q_2 \in C(M)$. If 
\[
C_{g,q_1} = C_{g,q_2},
\]
is it true that $q_1 = q_2$?
\end{QuestionNoNumber}

The answer is positive in two dimensions \cite{GT}, and there are partial results in special geometries when $n \geq 3$ \cite{DKSaU, DKLS}. Recall that the above question includes the case of isotropic conductivities: if $\gamma \in C^2(M)$ is positive, then inverse problems for the equation 
\[
\mathrm{div}_g(\gamma \nabla_g u) = 0 \text{ in $M$}
\]
can be reduced to the study of $C_{g,q}$ using the substitution $u = \gamma^{-1/2} v$.

The aim in this work is to show that from the knowledge of $C_{g,q}$ one can determine integrals of $q$ over certain two-dimensional manifolds. These manifolds are related to the notion of limiting Carleman weights studied in \cite{KSU, DKSaU}.

\begin{Definition}
Let $(\mO,g)$ be an open Riemannian manifold that contains $(M,g)$. A real valued function $\varphi \in C^{\infty}(\mO)$ is a limiting Carleman weight (LCW) in $\mO$ if $d\varphi$ is nonvanishing in $\mO$ and one has the Poisson bracket condition 
\[
\{ \bar{p}_{\varphi}, p_{\varphi} \} = 0 \text{ when $p_{\varphi} = 0$}
\]
where $p_{\varphi} \in C^{\infty}(T^* \mO)$ is the semiclassical Weyl principal symbol of the conjugated Laplacian $P_{\varphi} = e^{\varphi/h} (-h^2 \Delta_g) e^{-\varphi/h}$ in $\mO$.
\end{Definition}

Here $h > 0$ is a small parameter and we use the conventions of semiclassical microlocal analysis, see \cite{Zworski}. The notion of LCWs provides an extension of the method of complex geometrical optics solutions (dating back to \cite{SU}) to various more general situations. From the microlocal point of view, the LCW condition means that the conjugated Laplacian $P_{\varphi}$ is a semiclassical complex principal type operator and $p_{\varphi}^{-1}(0)$ is an involutive submanifold of $T^* \mO$. Complex involutive operators in the classical case (without the parameter $h$) have been studied in detail in \cite{DH72} and \cite[Section 26.2]{H}. In particular, the characteristic set is an involutive codimension $2$ submanifold of the cotangent space foliated by two-dimensional manifolds called bicharacteristic leaves. A complex involutive operator can be conjugated microlocally by Fourier integral operators to a normal form given by the Cauchy-Riemann operator $D_1 + i D_2$. Propagation of singularities occurs along bicharacteristic leaves, and this statement is sharp in the sense that for any suitable bicharacteristic leaf one can construct an approximate solution whose wave front set is on the leaf.

For our purposes, a bicharacteristic leaf is good if one can construct suitable approximate solutions (quasimodes) concentrating near the spatial projection of the leaf. This is related to a semiclassical version of the construction mentioned above, and will be possible for leafs that satisfy certain topological and nontrapping conditions. The next definition gives an abstract version of the required condition.

\begin{Definition} \label{def_good_bicharacteristic_leaf}
Let $(\mO,g)$ be an open Riemannian manifold containing $(M,g)$ and let $\varphi$ be an LCW in $\mO$. A bicharacteristic leaf $\Gamma \subset p_{\varphi}^{-1}(0)$ is \emph{good} for $\varphi$ if there is a compact simply connected set $K \subset \Gamma$, with $\Gamma \cap T^* M \subset K$, so that for any holomorphic function $\Psi$ near $K$ there are families $(w_{\pm}^{\Psi}(h)) \subset C^2(M)$ with $\norm{P_{\pm \varphi} w_{\pm}^{\Psi}}_{L^2(M)} = o(h)$ and $\norm{w_{\pm}^{\Psi}}_{L^2(M)} = O(1)$ as $h \to 0$, and for all holomorphic $\Psi, \Phi$ near $K$ 
\[
\lim_{h \to 0} \ (f w_+^{\Psi}, w_-^{\Phi})_{L^2(M)} = \int_{\Gamma} f \Psi \Phi \,dS, \qquad f \in C(M).
\]
More generally, $\Gamma$ is good if this holds for $\Psi, \Phi$ in some set $E$ of holomorphic functions near $K$ so that any holomorphic function near $K$ can be approximated uniformly on $K$ by functions $\Psi \Phi$ with $\Psi, \Phi \in E$.
\end{Definition}

\begin{RemarkNoNumber}
A few clarifications are in order. If $\varphi$ is an LCW in an open manifold $(\mO,g)$ containing $(M,g)$, then each bicharacteristic leaf in $p_{\varphi}^{-1}(0)$ has a natural complex structure \cite{DH72} as well as a Riemannian structure and volume form $dS$ induced by the Sasaki metric on $T^* \mO$. Moreover, a function $f$ on $M$ is identified with the function on $T^* \mO$ which only depends on the base point and vanishes outside $M$.
\end{RemarkNoNumber}

The following result reduces the Calder\'on problem into inverting a certain \emph{bicharacteristic leaf transform} involving integrals over good bicharacteristic leaves and holomorphic amplitudes.

\begin{Theorem} \label{thm_main1}
Let $(M,g)$ be a compact oriented Riemannian manifold with boundary, and let $q_1, q_2 \in C(M)$. If $C_{g,q_1} = C_{g,q_2}$, then 
\[
\int_{\Gamma} (q_1 - q_2) \Psi \,dS = 0
\]
whenever $\Gamma$ is a good bicharacteristic leaf for some LCW near $M$, and whenever $\Psi$ is holomorphic in $\Gamma$.
\end{Theorem}

\begin{ExampleNoNumber}
If $M \subset \mR^n$ and $g$ is the Euclidean metric, then each two-plane in $\mR^n$ gives rise to a good bicharacteristic leaf (associated with a linear LCW), and the integrals in Theorem \ref{thm_main1} include the integrals of $(q_1-q_2) \Psi$ over all two-planes in $\mR^n$, where $\Psi$ are holomorphic functions on the two-planes. In particular choosing $\Psi \equiv 1$ gives that the two-plane transform of $q_1-q_2$, understood as a compactly supported function in $\mR^n$, vanishes. If $n \geq 3$ this implies $q_1 \equiv q_2$ by the injectivity of the two-plane transform. Alternatively, if $n \geq 3$ one can obtain the vanishing of the Fourier transform of $q_1-q_2$ by superposing the integrals over parallel two-planes weighted with complex exponentials, which recovers the original argument of \cite{SU}.
\end{ExampleNoNumber}

We next discuss the existence of good bicharacteristic leaves. As mentioned above, in Euclidean space there are plenty of good bicharacteristic leaves. However, good bicharacteristic leaves are associated with LCWs which in turn require a conformal symmetry \cite{AFGR}. In particular, a generic manifold with $\dim(M) \geq 3$ does not admit any LCWs \cite{LS, A} and hence does not have any good bicharacteristic leaves. The following classes of manifolds that admit LCWs have been studied in \cite{DKSaU, DKLS}.

\begin{Definition}
A compact manifold $(M,g)$ with smooth boundary is \emph{transversally anisotropic} if $(M,g) \subset \subset (\mR \times M_0, e \oplus g_0)$ where $(M_0,g_0)$ is a compact $(n-1)$-dimensional manifold with boundary, called the \emph{transversal manifold}, and $(\mR,e)$ is the Euclidean line. We call $(M,g)$ \emph{conformally transversally anisotropic} (CTA) if $(M,cg)$ is transversally anisotropic for some smooth positive function $c$ in $\mR \times M_0$.
\end{Definition}

A function is an LCW near $(M,cg)$ if and only if it is an LCW near $(M,g)$, and if $c$ is known then $C_{cg,q}$ determines $C_{g,c(q-q_c)}$ where $q_c$ is known \cite{DKSaU}. Thus for present purposes it is sufficient to work with transversally anisotropic manifolds instead of CTA manifolds.

If $(M,g) \subset \subset (\mR \times M_0, e \oplus g_0)$ is transversally anisotropic, it has the natural LCW $\varphi(x) = x_1$ where $x_1$ is the coordinate along $\mR$. We now state the result regarding the existence of good bicharacteristic leaves. In part (b), recall that a manifold with boundary is \emph{nontrapping} if all geodesics reach the boundary in finite time, and \emph{strictly convex} if the second fundamental form of the boundary is positive definite.

\begin{Theorem} \label{thm_main2}
Let $(M,g)$ be compact with smooth boundary.
\begin{enumerate}
\item[(a)]
If $(M,g)$ is transversally anisotropic, almost every point of $M$ lies on at least one good bicharacteristic leaf.
\item[(b)]
If $(M,g)$ is transversally anisotropic with nontrapping strictly convex transversal manifold, then every bicharacteristic leaf for the natural LCW is good.
\item[(c)]
It is possible that $(M,g)$ admits an LCW but every bicharacteristic leaf for this LCW contains integral curves trapped in $M^{\mathrm{int}}$ (a possible obstruction for a leaf to be good).
\end{enumerate}
\end{Theorem}

If the set of good bicharacteristic leaves is nonempty, the next step is to determine how much information the integrals in Theorem \ref{thm_main1} contain. We will not address this question here except to remark that if $(M,g)$ is transversally anisotropic, then choosing holomorphic exponentials $e^{-2\lambda(s + it)}$ in Theorem \ref{thm_main1} where $\lambda \in \mR$ and $(s, t)$ are natural coordinates on the leaf $\Gamma$ implies one of the main results of \cite{DKLS}, namely that 
\[
\int_0^L e^{-2\lambda s} \left[ \int_{-\infty}^{\infty} e^{-2i\lambda t} (q_1-q_2)(t, \gamma(s)) \,dt \right] \,ds = 0
\]
for any transversal unit speed geodesic $\gamma: [0,L] \to M_0$ which is \emph{nontangential} in the sense that $\dot{\gamma}(0)$ and $\dot{\gamma}(L)$ are nontangential vectors on $\partial M_0$ and $\gamma(s) \in M_0^{\mathrm{int}}$ when $0 < s < L$. Thus as in \cite{DKLS}, whenever the geodesic X-ray transform on the transversal manifold is invertible, the vanishing of the integrals in Theorem \ref{thm_main1} implies that $q_1 = q_2$ (i.e.\ the bicharacteristic leaf transform is invertible).

Finally, as a byproduct of a geometric lemma required for Theorem \ref{thm_main2}(a), we show the invertibility of the geodesic X-ray transform on subdomains of product manifolds (the factors need to be non-closed, otherwise $M = [0,1] \times S^1$ provides a counterexample). This is a new condition for invertibility: we refer to \cite{PSU_survey, PSUZ} for further such conditions.

\begin{Theorem} \label{thm_main3}
Assume that $(M,g)$ is a compact subdomain with smooth boundary in the interior of $(M_1 \times M_2, g)$, where $(M_j,g_j)$ are non-closed manifolds and $g = g_1 \oplus g_2$. If $f \in C(M)$ integrates to zero over all maximal geodesics in $M$ joining boundary points, then $f \equiv 0$.
\end{Theorem}

The main result, Theorem \ref{thm_main1}, extends a number of results in earlier works such as \cite{SU, GU, DKSaU, DKLS} which can be interpreted as special cases. It remains to characterise which bicharacteristic leaves are good (besides the sufficient conditions discussed here), and to understand the invertibility properties of the bicharacteristic leaf transform (the argument outlined above for CTA manifolds reduces this question to the geodesic X-ray transform and this requires extra conditions on the transversal manifold).

However, the main point of the current article is the overall approach rather than the specific results that are stated. This approach highlights the role of microlocal normal forms, propagation of singularities and solutions that concentrate along submanifolds in solving the Calder\'on problem. Similar ideas certainly appear in earlier works as well. For instance, the complex involutive structure of the conjugated Laplacian plays a role in \cite{SU, KSU}. The works \cite{GU, DKSaU, DKLS} construct solutions concentrating near two-planes in $\mR^n$ and near more general two-dimensional manifolds in the geometric case, and \cite{GLSSU} uses the complex involutive structure for singularity detection in two dimensions. The approach outlined here is partly a reformulation of earlier ideas, but it provides further insight on the methods and suggests future directions.

This article is organized as follows. Section \ref{sec_intro} is the introduction. Section \ref{sec_concentrating} outlines the proof of Theorem \ref{thm_main1} and explains the basic ideas of the approach. In Section \ref{sec_nontrapping} we prove a geometric lemma that will be used for the discussion on transversally anisotropic manifolds in Section \ref{sec_transversally_anisotropic} and for the geodesic X-ray transform result in Section \ref{sec_xray}.

\subsection*{Notation}
We will mostly use the same notations as in \cite{DKSaU} and \cite{DKLS}. In particular the Riemannian geometry notation will be the same as in \cite[Appendix]{DKSaU}.

\subsection*{Acknowledgements}

The author would like to thank Victor Bangert for suggesting the proof of Lemma \ref{lemma_chord_covering}. The author was supported by the Academy of Finland (Centre of Excellence in Inverse Problems Research) and an ERC Starting Grant (grant agreement no 307023).

\section{Concentrating solutions in the Calder\'on problem} \label{sec_concentrating}

To explain the main ideas, we will outline the proof of Theorem \ref{thm_main1}. The short proof is made possible by the fact that the required properties were already assumed in the definition of a good bicharacteristic leaf. The real work lies in understanding the definition and showing that certain bicharacteristic leaves satisfy it.

\begin{proof}[Proof of Theorem \ref{thm_main1}]
The fact that $C_{g,q_1} = C_{g,q_2}$ and a standard integral identity imply that 
\[
( (q_1 - q_2) u_1, u_2)_{L^2(M)} = 0
\]
whenever $u_j \in H^1(M)$ satisfy $(-\Delta_g + q_1) u_1 = (-\Delta_g + \bar{q}_2) u_2 = 0$ in $M$. One would like to choose special solutions $u_j$ so that the above identity gives useful information about $q_1-q_2$. In particular, solutions with spatial concentration could be useful.

The notion of propagation of singularities provides a possible mechanism for finding solutions that concentrate. For instance, for real principal type operators one knows that singularities propagate along null bicharacteristic curves, and this is sharp in the sense that under a nontrapping condition one can construct an approximate solution with wave front set on such a curve \cite[Section 26.1]{H}. A semiclassical version of this construction, either using a Gaussian beam argument or conjugation by semiclassical Fourier integral operators into the related microlocal normal form $h D_1$, produces quasimodes that concentrate near the spatial projection of the bicharacteristic curve \cite{DKLS}.

The Calder\'on problem involves an elliptic equation and it is not immediately obvious how to produce concentrating solutions. However, conjugating the equation by exponentials reduces the question to complex involutive operators for which singularities do propagate if $n \geq 3$. In effect, if $\varphi \in C^{\infty}(M)$ is real valued, choosing $u_1 = e^{-\varphi/h} v_+$ and $u_2 = e^{\varphi/h} v_-$ leads to the identity 
\[
( (q_1 - q_2) v_+, v_-)_{L^2(M)} = 0
\]
whenever $v_{\pm} \in H^1(M)$ solve $(P_{\varphi} + h^2 q_1) v_+ = (P_{-\varphi} + h^2 \bar{q}_2) v_- = 0$ in $M$, where $P_{\varphi} = e^{\varphi/h} (-h^2 \Delta_g) e^{-\varphi/h}$. If $\varphi$ is an LCW, both $P_{\pm \varphi}$ are (semiclassical) complex involutive operators so singularities propagate along bicharacteristic leaves. Moreover, if $w_{\pm} \in C^2(M)$ satisfy
\[
\norm{P_{\pm \varphi} w_{\pm}}_{L^2(M)} = o(h), \qquad \norm{w_{\pm}}_{L^2(M)} = O(1) \text{ as $h \to 0$},
\]
then the solvability result \cite[Proposition 4.4]{DKSaU} based on Carleman estimates allows one to find solutions $v_{\pm} = w_{\pm} + o_{L^2(M)}(1)$ close to the quasimodes $w_{\pm}$. Using these solutions gives that 
\[
\lim_{h \to 0} \ ( (q_1 - q_2) w_+, w_-)_{L^2(M)} = 0.
\]

Now if $\Gamma$ is a good bicharacteristic leaf and $\Psi, \Phi$ are in some set $E$ of holomorphic functions as in Definition \ref{def_good_bicharacteristic_leaf}, one can choose $w_+ = w_+^{\Psi}(h)$ and $w_- = w_-^{\Phi}(h)$ and obtain that 
\[
\int_{\Gamma} (q_1-q_2) \Psi \Phi \,dS = 0.
\]
Since it was assumed that any holomorphic function near $K$ can be approximated uniformly in the set $K$ by functions $\Psi \Phi$ where $\Psi, \Phi \in E$, Theorem \ref{thm_main1} follows.
\end{proof}

We will next give a heuristic motivation for Definition \ref{def_good_bicharacteristic_leaf}. Let $\Gamma$ be a bicharacteristic leaf for $P_{\varphi}$, and consider the possibility of constructing approximate solutions (quasimodes) for $P_{\varphi}$ supported near the spatial projection of $\Gamma$ that satisfy the definition. One expects such quasimodes to exist for suitable leaves $\Gamma$ because the normal form for $P_{\varphi}$ is the Cauchy-Riemann operator: one could hope to find semiclassical Fourier integral operators $F, G$, with $F$ associated to the graph of a canonical transformation $\chi$ which straightens $\Gamma$ in phase space to a piece of $\mR^2$, so that (very roughly) 
\begin{equation} \label{conjugation_complex_global}
G P_{\varphi} F = hD_1 + i hD_2
\end{equation}
microlocally near $\Gamma$.

In the classical case this type of conjugation is possible microlocally near a fixed point of $T^* \mO$ \cite[Section 26.2]{H}, and for the simpler case of real principal type operators one can do this globally near a bicharacteristic curve \cite[Section 26.1]{H} even in the related semiclassical case \cite{DKLS} (see \cite[Chapter 12]{Zworski} for the semiclassical construction near a point). Assuming that \eqref{conjugation_complex_global} were possible microlocally near $\Gamma$, it would be easy to construct the quasimodes in Definition \ref{def_good_bicharacteristic_leaf} roughly by taking 
\[
w_+^{\Psi} = F(\Psi(x_1,x_2) m_+), \quad w_-^{\Phi} = F(\overline{\Phi(x_1,x_2)} m_-)
\]
where $m_{\pm}$ are quasimodes for $hD_1 \pm i h D_2$ so that $m_+ \overline{m_-}$ converges to the delta function of the relevant $2$-plane. These have the required limit profiles since at least for $f \in C^{\infty}_c(M^{\mathrm{int}})$, 
\begin{align*}
(f w_+^{\Psi}, w_-^{\Phi}) &= (f F (\Psi m_+), F(\overline{\Phi} m_-)) = (F^* f F(\Psi m_+), \overline{\Phi} m_-) \\
 &= ( (\chi^*f) \Psi \Phi m_+, m_-) + o(1) \to \int_{\Gamma} f \Psi \Phi \,dS
\end{align*}
as $h \to 0$ by the semiclassical Egorov theorem.

The argument above is certainly very heuristic since \eqref{conjugation_complex_global} has not been justified in the present case. However, \cite{DH72} gives a construction of approximate solutions whose wave front set lies on a given bicharacteristic leaf that satisfies a topological (trivial holonomy) and a nontrapping condition. It is plausible that there should be a semiclassical version of this construction, and this will be dealt with in a later version of this paper.

\section{Nontrapping properties} \label{sec_nontrapping}

The following basic geometric lemma will be used in the proof of Theorems \ref{thm_main2} and \ref{thm_main3}.

\begin{Lemma} \label{lemma_chord_covering}
Let $(M,g)$ be a compact manifold with strictly convex boundary. Then almost every point of $M$ lies on some nontangential geodesic between boundary points.
\end{Lemma}

Let $(M,g)$ be embedded in some closed manifold $(N,g)$, let $SM$ and $SN$ be the corresponding unit sphere bundles, let $\varphi_t$ be the geodesic flow on $SN$, and for $(x,v) \in SM$ let 
\begin{align*}
l_{+}(x,v) &= \sup \ \{ t \geq 0 \,;\, \varphi_t(x,v) \in SM \}, \\
l_{-}(x,v) &= \inf \ \{ t \leq 0 \,;\, \varphi_t(x,v) \in SM \}.
\end{align*}
Trapped geodesics correspond to the cases where $l_{\pm}(x,v) = \pm \infty$.

Consider the disjoint union 
\[
SM = G \cup B_1 \cup B_2
\]
where $G$ and $B_j$ are sets of good and bad directions, 
\begin{align*}
G &= \{Ê(x,v) \in SM \,;\, \text{both $l_+(x,v)$ and $l_-(x,v)$ are finite} \}, \\
B_1 &= \{Ê(x,v) \in SM \,;\, \text{exactly one of $l_{\pm}(x,v)$ is finite} \}, \\
B_2 &= \{Ê(x,v) \in SM \,;\, \text{both $l_+(x,v)$ and $l_-(x,v)$ are infinite} \}.
\end{align*}
Since $\{ l_{\pm} = \pm \infty \}$ are closed sets in $SM$, these sets are measurable. The idea is to show that $B_1$ is negligible. Since for any $x \in M$, some $(x,v)$ is in $G \cup B_1$, this will lead to the fact that for almost every $x \in M$ there is a good direction $(x,v)Ê\in G$.

We will denote by $m$ the volume measure on $(M,g)$, and by $\mu$ the Liouville measure on $SM$.

\begin{Lemma} \label{lemma_b_zero_measure}
$B_1$ has zero measure.
\end{Lemma}
\begin{proof}
Consider the sets 
\begin{align*}
K_+ &= \{Ê(x,v) \in SM \,;\, l_+(x,v) = \infty, \ l_-(x,v) \text{ finite} \}, \\
S &= \{Ê(x,v) \in SM \,;\, l_+(x,v) = \infty, \ \abs{l_-(x,v)} < 1 \}.
\end{align*}
If $t \geq 0$ one has 
\[
\varphi_t(S) = \{Ê(x,v) \in SM \,;\, l_+(x,v) = \infty, \ t \leq \abs{l_-(x,v)} < t + 1 \}
\]
and thus $K_+$ can be written as the disjoint union 
\[
K_+ = \bigcup_{k=0}^{\infty} \varphi_k(S).
\]
Since the Liouville measure is invariant under geodesic flow, one has 
\[
\mu(K_+) = \sum_{k=0}^{\infty} \mu(\varphi_k(S)) = \sum_{k=0}^{\infty} \mu(S).
\]
Now $SM$ is compact so $\mu(K_+) < \infty$, which implies that $\mu(S) = 0$. It follows that $\mu(K_+) = 0$ and $\mu(B_1) = 0$.
\end{proof}

In the following proof, if $U$ is a coordinate neighborhood in $M^{\mathrm{int}}$, we consider sets of the form 
\[
S_{V,W} = \{Ê(x,g(x)^{-1/2} \omega) \in SM \,;\, x \in V, \ \omega \in W \}
\]
where $V \subset U$ and $W \subset S^{n-1}$. Since the map $\omega \mapsto g(x)^{-1/2} \omega$ is an isometry from the round sphere $S^{n-1}$ onto $S_x M = \{ (x,v) \,;\, \abs{v} = 1 \}$ with the Sasaki metric, a local coordinate computation shows that 
\[
\mu(S_{V,W}) = m(V) m_{S^{n-1}}(W).
\]

\begin{proof}[Proof of Lemma \ref{lemma_chord_covering}]
Define 
\[
A = \{Êx \in M \,;\, \text{for any $v \in S_x M$, at least one of $l_{\pm}(x,v)$ is infinite} \}.
\]
We argue by contradiction and assume that $A$ has positive measure. Then by the Lebesgue density theorem there exists a point $x_0$ of density one in $A$, meaning that 
\[
\lim_{\eps \to 0} \frac{m(A \cap B_{\eps}(x_0))}{m(B_{\eps}(x_0))} = 1.
\]
Note that $A \subset M^{\mathrm{int}}$ since the boundary is strictly convex, thus also $x_0 \in M^{\mathrm{int}}$. Since any point of $M$ can be connected to the boundary by some geodesic (see e.g.\ \cite[Lemma 2.10]{KKL}), there is $v_0 \in S_{x_0} M$ with $l_+(x_0,v_0) < \infty$. By strict convexity the geodesic through $(x_0,v_0)$ exits $M$ nontangentially. The implicit function theorem shows that $l_+(x,v) < \infty$ for $(x,v)$ in some neighborhood of $(x_0,v_0)$. Thus there exist $\eps_0 > 0$ and an open set $W \subset S^{n-1}$ so that $l_+(x,v) < \infty$ for any $(x,v) \in S_{B_{\eps}(x_0),W}$ when $\eps < \eps_0$. It follows that 
\[
(x,v) \in S_{B_{\eps}(x_0),W} \cap SA \implies (x,v) \in B_1 \cap SA.
\]
But $\mu(B_1) = 0$ by Lemma \ref{lemma_b_zero_measure}. Thus, for $\eps < \eps_0$, it follows that  
\[
0 = \mu(S_{B_{\eps}(x_0),W} \cap SA) = \mu(S_{B_{\eps}(x_0) \cap A, W}) = m(B_{\eps}(x_0) \cap A) m_{S^{n-1}}(W).
\]
Since $m_{S^{n-1}}(W) > 0$, the last statement contradicts the fact that $x_0$ was a point of density one in $A$.
\end{proof}

\section{Transversally anisotropic manifolds} \label{sec_transversally_anisotropic}

In this section we will study bicharacteristic leaves on transversally anisotropic manifolds. However, we begin with some facts concerning the general case. See \cite{DKSaU} for further information.

Let $(M,g)$ be a compact oriented manifold with boundary, let $(\mO,g)$ be an open manifold containing $(M,g)$, and let $\varphi$ be an LCW in $(\mO,g)$. The conjugated Laplacian in $\mO$ is given by 
\[
P_{\varphi} = e^{\varphi/h}(-h^2 \Delta_g) e^{-\varphi/h}.
\]
The (semiclassical Weyl) principal symbol is the function $p_{\varphi}$ on $T^* \mO$ given by 
\[
p_{\varphi} = \abs{\xi}^2 - \abs{d\varphi}^2 + 2i \langle d\varphi, \xi \rangle.
\]
We write $p_{\varphi} = a + ib$ where $a = \abs{\xi}^2 - \abs{d\varphi}^2$ and $b = 2 \langle d\varphi,Ê\xi \rangle$ are the real and imaginary parts of $p_{\varphi}$. The characteristic set is given by 
\begin{align*}
p_{\varphi}^{-1}(0) &= \{Ê(x,\xi) \in T^* \mO \,;\, a(x,\xi) = b(x,\xi) = 0 \} \\
 &= \{Ê(x,\xi) \in T^* \mO \,;\, \abs{\xi} = \abs{d\varphi},\ \ \xi \perp d\varphi \}.
\end{align*}
Given any point $(x_0, \xi_0) \in p_{\varphi}^{-1}(0)$, the bicharacteristic leaf through $(x_0,\xi_0)$ is obtained by following the integral curves of the Hamilton vector fields $H_a$ and $H_b$. The next simple result describes these integral curves.

\begin{Lemma}
The integral curve of $H_a$ through $(x_0, \xi_0)$ is the cogeodesic $(x(t), \xi(t))$ where $\xi(t) = \dot{x}(t)^{\flat}$ and $x(t)$ is the geodesic 
\[
D_{\dot{x}(t)} \dot{x}(t) = 0, \ \ x(0) = x_0, \ \ \dot{x}(0) = \xi_0^{\sharp}.
\]
The integral curve of $H_b$ through $(x_0,\xi_0)$ is $(x(t),Ê\xi(t))$ where $x(t)$ is the integral curve of $2 \nabla \varphi$ through $x_0$ where $\nabla \varphi = \mathrm{grad}_g \varphi$, 
\[
\dot{x}(t) = 2 \nabla \varphi(x(t)), \ \ x(0) = x_0,
\]
and $\xi(t)$ is the parallel transport of $\xi_0$ along $x(t)$,
\[
D_{\dot{x}(t)} \xi(t) = 0, \ \ \xi(0) = \xi_0.
\]
\end{Lemma}

We now specialize to the case of transversally anisotropic manifolds and assume that $(M,g) \subset \subset (\mRÊ\times M_0, e \oplus g_0)$ where $(M_0,g_0)$ is compact with boundary. We may assume that $(M_0,g_0)$ is contained in an open manifold $(\mO_0, g_0)$ and $\mO = \mR \times \mO_0$, so that $\varphi(x) = x_1$ will be the natural LCW in $(\mO,g)$ where we write $(x_1,x')$ for coordinates in $\mR \times \mO_0$. In this setting we have 
\[
a = \abs{\xi}^2 - 1, \qquad b = 2 \xi_1
\]
and the characteristic set is $p_{\varphi}^{-1}(0) = \{Ê(x,\xi) \in T^* \mO \,;\, \xi_1 = 0, \abs{\xi} = 1 \}$. Given a point $(y,\eta)$ in the characteristic set, the bicharacteristic leaf through $(y,\eta)$ is given by 
\[
\Gamma_{y,\eta} = \{ \ ( \ (y_1 + t, \gamma(s)) , \ (0, \dot{\gamma}(s)) \ ) \,;\, s \in \mR, \ t \in (a,b) \}
\]
where $\eta_1 = 0$, $\abs{\eta'}_{g_0} = 1$, and $\gamma$ is the unit speed geodesic in $(\mO_0,g_0)$ through $(y', \eta')$ with maximal interval of existence $(a,b)$. The spatial projection of $\Gamma_{y,\eta}$ is the translation of the transversal geodesic $\gamma$ in the $x_1$ direction. Such a leaf turns out to be good if $\gamma$ is nontangential.

\begin{Lemma} \label{lemma_bicharacteristic_leaf_nontangential}
The bicharacteristic leaf $\Gamma_{y,\eta}$ is good whenever the part of $\gamma$ that lies in $M_0$ is a nontangential geodesic in $(M_0,g_0)$.
\end{Lemma}
\begin{proof}
After reparametrizing $\gamma$, suppose that $\gamma|_{[0,L]}$ is the part of $\gamma$ that lies in $M_0$ and that this part is nontangential. Fix $\lambda_{\pm} \in \mR$, let $h > 0$ be small, and define the quasimodes $w_{\pm} \in C^2(M)$, 
\[
w_+ = e^{-i\lambda_+ x_1} v_+(x'), \qquad w_- = e^{i\lambda_- x_1} v_-(x')
\]
where $v_{\pm} = v_{h^{-1}+i\lambda_{\pm}} \in C^{\infty}(M_0)$ are the functions given in \cite[Theorem 1.7]{DKLS} that satisfy, for any fixed $N > 0$, 
\begin{gather*}
\norm{(-\Delta_{g_0} - (h^{-1}+i\lambda_{\pm})^2) v_{\pm}}_{L^2(M_0)} = O(h^N), \qquad \norm{v_{\pm}}_{L^2(M_0)} = O(1), \\
\int_{M_0} v_+ \bar{v}_- \psi \,dV_{g_0} \to \int_0^L e^{-(\lambda_+ + \lambda_-) s} \psi(\gamma(s)) \,ds
\end{gather*}
as $h \to 0$ for any $\psi \in C(M_0)$ (the last fact is only stated for $\lambda_+ = \lambda_-$ but the proof works also in the above case). It follows that 
\[
\norm{P_{\pm \varphi} w_{\pm}}_{L^2(M)} = o(h), \qquad \norm{P_{\pm \varphi} w_{\pm}}_{L^2(M)} = O(1)
\]
as $h \to 0$, and 
\[
\lim_{h \to 0} \int_M f w_+ \bar{w}_- \,dV_g = \int_{\mR^2} f(t, \gamma(s)) e^{-\lambda_+(s+it)} e^{-\lambda_-(s+it)} \,ds \,dt.
\]

The conditions in Definition \ref{def_good_bicharacteristic_leaf} will be satisfied with 
\[
K =  \{ \ ( \ (y_1 + t, \gamma(s)) , \ (0, \dot{\gamma}(s)) \ ) \,;\, s \in [0,L], \ t \in I \}
\]
where $I$ is a suitable compact interval, and with holomorphic functions $\Phi, \Psi \in E = \{ e^{-\lambda(s+it)} \,;\, \lambda \in \mR \}$. The density statement for $E$ follows since if $F$ is holomorphic near $K$, then $F|_K$ is in the $C(K)$ closure of the set $\{ \Psi \Phi|_K \,;\, \Psi, \Phi \in E \}$: if $\mu$ is a complex measure in $K$ integrating to zero against $e^{-\lambda(s+it)}|_K$ for all $\lambda \in \mR$, then differentiating in $\lambda$ shows that $\mu$ integrates to zero against complex polynomials, and by Runge's theorem $\mu$ integrates to zero against $F|_K$.
\end{proof}

\begin{proof}[Proof of Theorem \ref{thm_main2}]
(a) After possibly enlarging the transversal manifold $(M_0,g_0)$, we may assume that $(M_0,g_0)$ has strictly convex boundary (to achieve this embed $(M_0,g_0)$ into some closed manifold, remove a small neighbourhood of a point not in $M_0$ and glue a part with strictly convex boundary). Lemma \ref{lemma_chord_covering} ensures that almost every point in $(M_0,g_0)$ lies on some nontangential geodesic, and by Lemma \ref{lemma_bicharacteristic_leaf_nontangential} the corresponding bicharacteristic leaf is good.

(b) If the transversal manifold is nontrapping with strictly convex boundary, then all maximal transversal geodesics are nontangential and by Lemma \ref{lemma_bicharacteristic_leaf_nontangential} the corresponding bicharacteristic leaves are good.

(c) Consider the set $\mO = \mR \times S^{n-1}$, $n \geq 3$, with coordinates $(t,y)$. We will define a compact submanifold $M$ of $\mO$ with smooth boundary as follows. Fix a small $\eps > 0$, and define 
\[
S(\omega) = \{ y \in S^{n-1} \,;\, y \cdot \omega \leq 1-\eps \}, \quad \omega \in S^{n-1}.
\]
Let $f: \mR \to S^{n-1}$ be a smooth map with $f(t) = e_n$ for $t \leq \eps$ and $t \geq 1-\eps$ but with $f(1/2) = e_1$, and define 
\[
N = \{ (t,y) \in \mO \,;\ \ y \in S(f(t)) \}.
\]
The manifold $\{Ê(t,y) \in N \,;\, t \geq 1 \}$ is isometric via stereographic projection to $([1,\infty) \times B_R, e \oplus g)$ where $B_R$ is a closed ball in $\mR^{n-1}$ and $g$ corresponds to the metric on $S^{n-1}$. One can thus work in $\mR^n$ and add a suitable cap to $\{Ê(t,y) \in N \,;\, t \in [0,1] \}$ when $t > 1$, and similarly when $t < 0$, to obtain a compact submanifold $M$ of $\mO$ with smooth boundary.

Consider the function $\varphi: \mO \to \mR, \ \varphi(t,y) = t$. This is an LCW in the open manifold $\mO$ containing $M$, and the bicharacteristic leaves are of the form 
\[
\{Ê(t, \gamma(s) ; 0, \dot{\gamma}(s)) \,;\, \text{$\gamma$ is a geodesic in $S^{n-1}$} \}.
\]
Also, the integral curve of $H_a$ through $(t,\omega ; 0,\eta)$ is the curve $s \mapsto (t,\gamma(s); 0, \dot{\gamma}(s))$ where $\gamma$ is the geodesic in $S^{n-1}$ with $\gamma(0) = \omega$ and $\dot{\gamma}(0) = \eta$.

Suppose that $\Gamma$ is a bicharacteristic leaf in $T^* \mO$. When $t=0$, $\Gamma$ contains a point of the form $(0,e_2;0,\eta)$ where $\eta \cdot e_2 = 0$. Now the integral curve of $H_a$ through this point will be trapped inside $M$ unless $\eta \approx e_n$ if $\eps$ is small. But then also the point $(1/2, e_2 ; 0, \eta)$ is in $\Gamma$, and the integral curve of $H_a$ through this point will be trapped inside $M$.
\end{proof}

\section{Geodesic X-ray transform on product manifolds} \label{sec_xray}

Finally, we prove the invertibility result for the geodesic X-ray transform on product manifolds.

\begin{proof}[Proof of Theorem \ref{thm_main3}]
First we show that $M_1$ (and similarly $M_2$) can be assumed to be a compact manifold with strictly convex boundary. Since $M_1$ is not closed, either it is compact with nonempty boundary or it is noncompact. In the first case we may embed $M_1$ in some closed manifold, remove a small neighborhood of some point which is not in $M_1$ and glue a part with strictly convex boundary near the removed part. In the second case, since $M \subset K^{\mathrm{int}} \subset K \subset M_1 \times M_2$ where $M$ and $K$ are compact, the projections $E = \{Êx_1 \,;\, (x_1,x_2) \in M \}$ and $L = \{ x_1 \,;\, (x_1,x_2) \in K \}$ are compact and satisfy 
\[
E \subset L^{\mathrm{int}} \subset L \subset M_1.
\]
Let $\eps = d(E, M_1 \setminus L^{\mathrm{int}} ) > 0$ and let $r: M_1 \to \mR, \ r(x) = d(x,E)$ be a distance function. Since $M_1$ is noncompact, one can find a smooth approximation $r_1$ of $r$ so that $E \subset \{ r_1 < \eps/2 \} \subset \{ r_1 \leq \eps/2 \} \subset L^{\mathrm{int}}$, and $dr_1 \neq 0$ on $\{ r_1 = \eps/2 \}$. We may replace $M_1$ by the compact manifold with boundary $\{ r_1 = \eps/2 \}$, and then use the argument above to replace $M_1$ by a compact manifold with strictly convex boundary.

Thus assume that $(M,g) \subset \subset (M_1 \times M_2, g)$ where $g = g_1 \oplus g_2$ and $(M_j,g_j)$ are compact with strictly convex boundary. By Lemma \ref{lemma_chord_covering}, there is a set $A_j$ of full measure in $M_j$ so that any fixed point $(x_1, x_2)$ in $A_1 \times A_2$ lies on some finite length unit speed geodesic $\gamma_j: [0,T_j]Ê\to M_j$ between boundary points. We may extend $M_j$ and its geodesic vector field $X_j$ to a larger manifold $N_j$ so that $X_j$ will have complete flow in $S N_j$ and the integral curves never return to $M_j$ once they exit $M_j$ (see \cite[Section 2]{G}). Define a curve in $N_1 \times N_2$ by 
\[
\gamma: \mR \to N_1 \times N_2, \ \ \gamma(t) = (\gamma_1(t), \gamma_2(t))
\]
where $\gamma_j$ are extended from the interval $[0,T_j]$ to $\mR$ as the spatial projections of integral curves of $X_j$. If $f$ is extended by zero to $N_1Ê\times N_2$, the fact that $f$ integrates to zero over all maximal geodesics in $(M,g)$ between boundary points implies that 
\[
\int_{-\infty}^{\infty} f(\gamma(t)) \,dt = 0.
\]
In fact the integrand is zero for all $t$ for which $\gamma(t)$ is outside $M$, and for the interval where $\gamma(t)$ is inside $M$ the curve $\gamma(t)$ is a unit speed maximal geodesic and the integral over that interval is also zero.

We can generate new curves in $N_1 \times N_2$ as follows. Consider 
\[
\eta: \mR \to N_1 \times N_2, \ \ \eta(t) = (\gamma_1(t \cos \theta + a_1), \gamma_2(t \sin \theta + a_2))
\]
where $\theta \in \mR$ and $a = (a_1, a_2) \in \mR^2$. Again, this curve is only inside $M$ for some compact interval and in that interval $\eta(t)$ is a unit speed maximal geodesic in $(M,g)$. It follows that 
\[
\int_{-\infty}^{\infty} f(\gamma_1(t \cos \theta + a_1), \gamma_2(t \sin \theta + a_2)) \,dt = 0.
\]
Writing $h(y_1,y_2) = f(\gamma_1(y_1), \gamma_2(y_2))$, this implies that 
\[
\int_{-\infty}^{\infty} h(t \omega + a) \,dt = 0
\]
for all $\omega \in S^1$ and $a \in \mR^2$. This shows that the two-dimensional Radon transform of $h$ vanishes, which implies that $h \equiv 0$. Consequently $f(x_1,x_2) = 0$ whenever $(x_1, x_2) \in A_1 \times A_2$, which implies that $f \equiv 0$ since the last set has full measure in $M_1 \times M_2$.
\end{proof}

\bibliographystyle{alpha}

\end{document}